\documentclass[11.7pt]{article}
\usepackage[utf8]{inputenc}
\usepackage{amsmath}
\usepackage{amsfonts}
\usepackage{amssymb}
\usepackage[symbol]{footmisc}
\usepackage[dvipsnames]{xcolor}
\usepackage{amsthm}
\usepackage{bbm}

\newtheorem{theorem}{Theorem}[section]
\newtheorem{lemma}[theorem]{Lemma}
\newtheorem{proposition}[theorem]{Proposition}
\newtheorem{corollary}[theorem]{Corollary}
\newtheorem{conjecture}[theorem]{Conjecture}
\theoremstyle{definition}

\newtheorem{remark}[theorem]{Remark}

\newtheorem{question}[theorem]{Question}

\newcommand{\R}{{\mathbb{R}}}
\newcommand{\C}{{\mathbb{C}}}
\newcommand{\N}{{\mathbb{N}}}

\newcommand{\ehzcap}{c_{_{\rm EHZ}}}

\newcommand{\dg}{\dot{\gamma}}
\newcommand{\dz}{\dot{z}}

\newcommand{\ceil}[1]{\lceil{#1}\rceil}
\newcommand{\floor}[1]{\lfloor{#1}\rfloor}

\def\vol{{\rm Vol}}

\begin{document}
	\title {Remarks on symplectic capacities of $p$-products}
	\author{P. Haim-Kislev, Y. Ostrover}
	\maketitle
	\begin{abstract}  
	In this note we study the behavior of symplectic capacities of convex domains in the classical phase space with respect to symplectic $p$-products.
As an application, by using a ``tensor power trick", we show that it is enough to prove  the weak version of Viterbo's  volume-capacity conjecture in the asymptotic regime, i.e., when the dimension is sent to infinity.  In addition, we introduce a conjecture about higher-order capacities of $p$-products, and show that if it holds, then there are no non-trivial $p$-decompositions of the symplectic ball.
	\end{abstract}
	
\section{Introduction and Results}

Symplectic capacities are numerical invariants that serve as a fundamental tool in the study of symplectic rigidity phenomena. Most known examples are closely related to Hamiltonian dynamics, and can moreover be used to study the 
existence and the behavior of periodic orbits of certain Hamiltonian systems. We refer the reader to~\cite{CHLS, McD1} 
for more information on symplectic capacities, their properties, and the role that they play in symplectic topology.

In this note we consider symplectic capacities of convex subsets of the classical phase space $({\mathbb R}^{2n},\omega_n)$, and study their behavior with respect to ``symplectic $p$-products" (cf.~\cite{Ke-Li,OR}). More precisely, recall that the Cartesian product $M \times N$ of two symplectic manifolds $(M,\omega_M)$, $(N, \omega_N)$
has a natural symplectic structure given by $\omega = \pi_M^* \omega_M + \pi_N^* \omega_N$, where $\pi_M,\pi_N$ are the two natural projections. 
For two convex bodies $K \subset {\mathbb R}^{2n}$ and $T \subset {\mathbb R}^{2m}$, which contain the origin of their respective ambient spaces, and $1 \leq p \leq \infty$, we consider the well-known $p$-product operation 
\begin{equation*} K \times_p T := \bigcup_{0 \leq t \leq 1} \Bigl (   (1-t)^{1/p}K \times t^{1/p}T \Bigr) \subset {\mathbb R}^{2n} \times {\mathbb R}^{2m}. \end{equation*}
Since the above definition of the $p$-product of two bodies is naturally applicable only when the bodies contain the origins of their respective ambient spaces, from now on we will assume that all convex bodies contain the origin unless specifically stated otherwise.
Note that $K \times_\infty T = K \times T$, and that $K \times_1 T = {\rm conv} \{(K \times 0) \cup  ( 0 \times T) \}$, where ${\rm conv}$ stands for the convex hull of a set. The 1-product $ K \times_1 T$ is also known as the {\it free sum} of $K$ and $T$, and is related to the product operation via duality (see e.g.,~\cite{HRGZ}).  We remark that 
the norm naturally associated with $K \times_p T$ satisfies
$$ \|(x,y)\|^p_{K \times_p T}  = \|x\|^p_K + \|y\|^p_T, \ \ {\rm and} \ \ \|(x,y)\|_{K \times_\infty T} = \max \{ \|x\|_K, \|y\|_T \}.$$
Moreover, let 
$h_L : {\mathbb R}^{d} \rightarrow {\mathbb R}$ be the support function associated with a convex body $L$ in ${\mathbb R}^d$ (see the notations given at the end of this section). Then, the support function of $K \times_p T$  satisfies
	$$ h_{K \times_p T}(x,y)^q = h_K(x)^q + h_T(y)^q, \ \ {\rm and} \ \ h_{K \times_1 T} = \max \{ h_K(x), h_T(y) \}, $$
	where $\frac{1}{p} + \frac{1}{q} = 1$ and $p,q \geq 1$.

\subsection{The systolic ratio of symplectic $p$-products}

For a convex domain $K \subset {\mathbb R}^{2n}$, it is known that many normalized symplectic capacities, including the first Ekeland-Hofer capacity $c_{\rm EH}^1$~\cite{EH1,EH2}, the Hofer-Zehnder capacity $c_{\rm HZ}$~\cite{HZ1}, the symplectic homology capacity $c_{\rm SH}$~\cite{V2}, and the first Gutt-Hutchnigs capacity $c^1_{\rm GH}$~\cite{Gu-Ha}, coincide. Moreover, when $K$ is smooth, all these capacities are given  
by the minimal action among all the closed characteristics on the boundary $\partial K$\footnote[1]{If the boundary of $K$ is not smooth, the above capacities coincide with the minimal action among ``generalized closed characteristics", as explained, e.g., in~\cite{AA-O1}.}. The above claims follow from a combination of results from~\cite{Ab-Ka,Gi-Sh, Gu-Ha, HZ1,Ir}. In what follows, for a convex domain $K$ in  ${\mathbb R}^{2n}$ we shall denote the above mentioned coinciding capacities by $\ehzcap(K)$.
The systolic ratio $$ {sys_n}(K) := {\frac {\ehzcap(K)} {(n! \vol(K))^{1/n}} } $$ of $K$ is defined as the ratio of this capacity of $K$ to the normalized $\omega$-volume of $K$. Note that $sys_n(B) = 1$, for any Euclidean ball $B$ in ${\mathbb R}^{2n}$. 

Recall the following weak version of Viterbo's volume-capacity conjecture~\cite{V1}. 
\begin{conjecture}[Viterbo] 
	\label{conj-viterbo}
	If $K \subset {\mathbb R}^{2n}$ is a convex domain, then 
	$$ sys_n(K) \leq 1. $$
\end{conjecture}

Our first result concerns the systolic ratio of symplectic $p$-products.
We show that if two convex bodies $K \subset {\mathbb R}^{2n}$ and $T \subset {\mathbb R}^{2m}$ fulfill Conjecture~\ref{conj-viterbo}, then the same is true for the $p$-product of $K$ and $T$. More precisely, 
\begin{theorem} \label{thm-p-product-systolic}
	For convex bodies $K \subset {\mathbb R}^{2n}$, $T \subset {\mathbb R}^{2m}$, and $1 \leq p \leq \infty$ 
	$$ sys_{n+m}(K \times_p T)^{m+n} \leq sys_n(K)^n sys_m(T)^m, $$
	 where equality holds if and only if $\ehzcap(K) = \ehzcap(T)$ and $p=2$.
\end{theorem}

\begin{remark}
    It is not hard to deduce that $sys_{n+m}(K \times_p T) < sys_{n+m}(K \times_2 T)$ for $p > 2$ from the proof of Theorem \ref{thm-p-product-systolic} below.
    On the other hand, for $p = 1$, consider a convex body $K \subset \R^{2n}$ and choose a convex body $T \subset \R^2$ with $c(T) = n c(K).$ Then, using Proposition \ref{prop-p-product-EHZ-capacity} below, one gets that 
    $$ \frac{sys_{n+1}(K \times_1 T)}{sys_{n+1}(K \times_2 T)} = \left( \frac{n}{n+1} \right) (2n+1)^\frac{1}{n+1} > 1, \ \ \ {\rm for} \ n \geq 2.$$
\end{remark}

One application of Theorem~\ref{thm-p-product-systolic} is a ``tensor power trick" (see, e.g., the discussion in Section 1.9.4 of~\cite{Ta}), which shows that it is enough to prove Conjecture~\ref{conj-viterbo} above for convex domains in the asymptotic regime $n \to \infty$.

\begin{corollary} \label{cor-asymp-Viterbo} 
	If Conjecture~\ref{conj-viterbo} holds in dimension $2n$ for some $n>1$, then it also holds in dimension $2m$ for every $m\leq n$. Moreover, if there exists a sequence $\alpha(n) \xrightarrow{n \to \infty } 1$ such that for every convex body $K \subset \R^{2n}$ one has 
	$$ sys_n(K) \leq \alpha(n) ,$$ 
	then Conjecture~\ref{conj-viterbo} holds in every dimension $n$. 
\end{corollary}
We remark that Corollary~\ref{cor-asymp-Viterbo} was already known to experts~\cite{Hut}.
Another application of Theorem~\ref{thm-p-product-systolic} concerns obstructions on 
the possibility of a ``$p$-product decomposition" of the Euclidean ball. This will be discussed in Section~\ref{sec-symp-decompositios} below (see Corollary~\ref{Cor-p-decomposition-of-the-ball-systolic}).

It is well known (see e.g.,~\cite{CHLS}) that $\ehzcap$ satisfies the Cartesian product property i.e., for two convex bodies  $K \subset {\mathbb R}^{2n}$ and $T \subset {\mathbb R}^{2m}$,
$$ \ehzcap (K \times T) = \min \{ \ehzcap (K ),\ehzcap ( T) \}.$$
An important ingredient in the proof of Theorem~\ref{thm-p-product-systolic} above, which might be also of independent interest, is the following generalization of this formula for the symplectic $p$-product of two convex bodies (cf. Lemma~\ref{lem_vol} below).

\begin{proposition}
    \label{prop-p-product-EHZ-capacity}
    For two convex bodies $K \subset {\mathbb R}^{2n}$, $T \subset {\mathbb R}^{2m}$, and $1 \leq p \leq \infty$,
	$$ \ehzcap(K \times_p T) = 
	\begin{cases}
	\min \{\ehzcap(K),\ehzcap(T)\}, & \text{if } \ 2 \leq p \leq \infty,\\
	\left(\ehzcap(K)^\frac{p}{p-2} + \ehzcap(T)^\frac{p}{p-2}\right)^{{\frac {p-2} p }}, & \text{if } \ 1 \leq p < 2.
	\end{cases}$$
\end{proposition}

\subsection{Higher-order capacities of symplectic $p$-products}

Consider the family of  Ekeland-Hofer capacities $\{c_{\rm EH}^k\}_{k=1}^{\infty}$ defined in~\cite{EH1,EH2}. We recall that in Eq.\,(3.8) from~\cite{CHLS} it is  asserted  that if $X_1 \subset {\mathbb R}^{2n}$ and $X_2 \subset {\mathbb R}^{2m}$ are 
two compact star-shaped domains, then 
\begin{equation} \label{formula-product-HE-first-capacity} c_{\rm EH}^k(X_1 \times X_2) = \min_{i+j=k} \{ c_{\rm EH}^i(X_1) + c_{\rm EH}^j(X_2)\}, \end{equation}
where $i$ and $j$ are non-negative integers, and by definition $c_{\rm EH}^0=0$.
Motivated by Proposition~\ref{prop-p-product-EHZ-capacity}, we conjecture the following generalization of~$(\ref{formula-product-HE-first-capacity})$.

\begin{conjecture}
	\label{Conj-EH-cap-of-p-products}
 For  star-shaped domains $X_1 \subset {\mathbb R}^{2n}$, $X_2 \subset {\mathbb R}^{2m}$, and $1 \leq p$,
	$$ c^k_{\rm EH}(X_1 \times_p X_2) = 
	\begin{cases}
	\min\limits_{i+j = k} \left[ c^i_{\rm EH}(X_1)^{\frac{p}{p-2}} + c^j_{\rm EH}(X_2)^{\frac{p}{p-2}} \right]^\frac{p-2}{p},& \text{if }p \geq 2,    \\
	\max\limits_{\substack{i+j= k+1 \\
                   i,j \neq 0}}  \left[ c^i_{\rm EH}(X_1)^{\frac{p}{p-2}} + c^j_{\rm EH}(X_2)^{\frac{p}{p-2}} \right]^\frac{p-2}{p},& \text{if }1 \leq p \leq 2.
	\end{cases}$$
\end{conjecture}

\begin{remark} \label{rmk-about-Conj}
    {\rm (I) In the case $p=2$, the two expressions in Conjecture~\ref{Conj-EH-cap-of-p-products} agree, and a combinatorial argument shows that they are equal to the $k^{\text{th}}$ term in the sequence $\{c_{\rm EH}^i(X_1),c_{\rm EH}^j(X_2)\}$ arranged in non-decreasing order with repetitions. For completeness, we provide a proof of this fact in the Appendix of the paper (see Lemma~\ref{formulas_coincide}). 
    
            (II) It is conjectured that the family $\{c^k_{\rm GH} \}_{k=1}^{\infty}$ of  Gutt-Hutchings symplectic capacities  introduced in~\cite{Gu-Ha} coincide with the family of Ekeland-Hofer capacities $\{c_{\rm EH}^k\}_{k=1}^{\infty}$ on the class of compact star-shaped domains of ${\mathbb R}^{2n}$ (see Conjecture 1.9 in~\cite{Gu-Ha}). Thus, Conjecture~\ref{Conj-EH-cap-of-p-products} above could be stated for $\{c^k_{\rm GH} \}_{k=1}^{\infty}$ as well.
    
    }
    \end{remark}

Using the remarkable fact that for convex/concave toric domains there are ``combinatorial" formulas to compute the Gutt-Hutchings capacities $\{c^k_{\rm GH} \}_{k=1}^{\infty}$ (see Theorems 1.6 and 1.14 in~\cite{Gu-Ha}), we can show that Conjecture~\ref{Conj-EH-cap-of-p-products}, when stated for the family $\{c^k_{\rm GH} \}_{k=1}^{\infty}$, holds in some special cases. More precisely,

\begin{theorem}
    	\label{GH_p_formula}
    	For convex toric domains $K \subset {\mathbb R}^{2n}$, $T \subset {\mathbb R}^{2m}$, and $p \geq 2$\footnote[2]{See Remark~\ref{rmk-about-Conj} (I) above about the case $p=2$.}
    		$$ c^k_{\rm GH}( K \times_p T) = \min_{i+j = k} \left[ c^i_{\rm GH}(K)^{\frac{p}{p-2}} + c^j_{\rm GH}(T)^{\frac{p}{p-2}} \right]^\frac{p-2}{p}.$$
    			For concave toric domains $K \subset {\mathbb R}^{2n}$, $T \subset {\mathbb R}^{2m}$, and $1 \leq p \leq 2$ 
    			$$ c^k_{\rm GH}( K \times_p T) = 	\max\limits_{\substack{i+j= k+1 \\
                   i,j \neq 0}}  \left[ c^i_{\rm GH}(K)^{\frac{p}{p-2}} + c^j_{\rm GH}(T)^{\frac{p}{p-2}} \right]^\frac{p-2}{p}.$$ 
 \end{theorem}
Theorem \ref{GH_p_formula} generalizes certain calculations of $c^k_{\rm GH}$ in some known cases. In particular, it is known that the Cartesian product property holds for convex toric domains (see Remark 1.10 in \cite{Gu-Ha}), and recently the Gutt-Hutchings capacities of $p$-product of discs were calculated in \cite{Ke-Li} (see also~\cite{OR}).

We end this subsection with a result about the limit of the normalized Ekeland-Hofer capacities. 
Following \cite{CHLS}, we set 
\[ \bar{c_k} := \frac{c^k_{\rm EH}}{k}, \]
and denote the limit of $\bar{c_k}$ as $k\to \infty$ by $c_\infty$. Denote also $c^{\rm GH}_\infty := \lim_{k \to \infty} \frac{c^k_{\rm GH}}{k}$ to be the same limit for the Gutt-Hutchings capacities.
In \cite[Problem 17]{CHLS} the authors raise the question whether $c_\infty$ coincides with the Lagrangian capacity $c_L$ defined in \cite{Ci-Mo}. In \cite{Gu-Ha} it is shown that $c^{\rm GH}_\infty(K) \leq c_L(K)$ for a convex or concave toric domain $K$, and moreover, in this case $c^{\rm GH}_\infty(K) = c_\square(K)$, where the cube capacity $c_\square(K)$ is the largest area of the 2-faces of a cube that can be symplectically embedded into $K$.
In \cite{CHLS}, it was proved that
\[ c_\infty(E(a_1,\ldots,a_n)) = \frac{1}{ 1/a_1 + \cdots + 1/a_n}, \]
for the ellipsoids 
$ E(a_1,\ldots, a_n) := \{z \in \C^n : \frac{\pi |z_1|^2}{a_1} + \cdots + \frac{\pi |z_n|^2}{a_n} < 1 \},$ and that
\[ c_\infty(P(a_1,\ldots, a_n)) = \min \{a_1, \ldots, a_n\} , \]
for the polydiscs $ P(a_1,\ldots, a_n) := B^2[a_1] \times \cdots \times B^2[a_n],$
where $B^n[r] \subset \R^{n}$ is the Euclidean ball of radius $\sqrt{r/\pi}$.
A generalization of the above formulas, assuming Conjecture \ref{Conj-EH-cap-of-p-products} holds, is the following.
\begin{theorem}\label{thm_inf_cap}
	Assume Conjecture \ref{Conj-EH-cap-of-p-products} holds. Then, for $1 \leq p \leq \infty$, and convex domains $K_1,\ldots,K_m \subset \R^{2n}$ such that $c_\infty(K_1),\ldots,c_\infty(K_m)$ exist, one has
	\begin{equation} \label{eq-c-inf}
	     c_\infty(K_1 \times_p \cdots \times_p K_m) = \left(c_\infty(K_1)^\frac{-p}{2} + \cdots +  c_\infty(K_m)^\frac{-p}{2}\right)^\frac{-2}{p}.
	\end{equation}
\end{theorem}
\begin{remark}
	{\rm (I) As ellipsoids and polydiscs are $p$-product of discs with $p=2$ and $p=\infty$, respectively, Theorem \ref{thm_inf_cap} does indeed generalize the above mentioned formulas for $c_\infty(E(a_1,\ldots, a_n))$ and $c_\infty(P(a_1,\ldots, a_n))$.
	
	(II) Theorem \ref{GH_p_formula} gives us that \eqref{eq-c-inf} holds for $c^{\rm GH}_\infty$, and for convex/concave toric domains with $p \geq 2$ and $1 \leq p \leq 2$, respectively. Moreover, this can be easily recovered by using the fact, shown in \cite{Gu-Ha}, that for convex/concave toric domains $$c^{\rm GH}_\infty(X_\Omega) = c_\square(X_\Omega) = \frac{1}{\|(1,\ldots,1)\|_\Omega}$$ (see Section \ref{sec-GH-of-p-products} below for the relevant definitions and notations).

	}
\end{remark}

\subsection{Symplectic $p$-decomposition of convex bodies } \label{sec-symp-decompositios}

In this section we consider the following general question:
\begin{question} \label{question-about-decomposition}
	 Which convex bodies are symplectomorphic to a symplectic $p$-product configuration of convex bodies of lower dimensions?
 \end{question}

An immediate corollary of Theorem~\ref{thm-p-product-systolic} is the following claim that relates the $p$-decomposition property of symplectic images of the Euclidean ball and Viterbo's volume-capacity conjecture.
\begin{corollary}
 \label{Cor-p-decomposition-of-the-ball-systolic}  
	If for some symplectic image of the Euclidean ball $\widetilde{B}$ in ${\mathbb R}^{2n+2m}$ there exist convex bodies $X \subset {\mathbb R}^{2n}$ and $Y \subset {\mathbb R}^{2m}$ such that $\widetilde{B} = X \times_p Y$ for some $p \neq 2$, then Conjecture~\ref{conj-viterbo} is false.
	Moreover, if $\widetilde{B} = X \times_2 Y$, and $sys_n(X) \neq 1$ or $sys_m(Y) \neq 1$, then  Conjecture~\ref{conj-viterbo} is false.
\end{corollary}

We finish the Introduction  with the following consequence of Conjecture~\ref{Conj-EH-cap-of-p-products} that answers Question~\ref{question-about-decomposition} for the case of the Euclidean ball.
\begin{theorem} \label{thm-application-of-conj-p-products} 
		Assume  Conjecture \ref{Conj-EH-cap-of-p-products} holds. Then, for 
		any pair of convex bodies $X\subset \R^{2n}, \, Y \subset \R^{2m}$,  and  $p\neq2$, the product $X \times_p Y$ is not symplectomorphic to a Euclidean ball.
		Moreover, 
		if a symplectic image of the ball $\widetilde{B}^{2(n + m)}[r] \subset \R^{2(n+m)}$ can be written as $\widetilde{B}^{2(n + m)}[r] = X \times_2 Y$ for some convex bodies $X \subset \R^{2n}, Y \subset \R^{2m}$, then one has $c^k_{\rm EH}(X) = c^k_{\rm EH}(B^{2n}[r])$ and $c^k_{\rm EH}(Y) = c^k_{\rm EH}(B^{2m}[r])$ for every $k \geq 1$.
\end{theorem}

\noindent{\bf Notations:} Denote by ${\mathcal K^m}$ 
the class of convex bodies in ${\mathbb R}^m$, 
i.e., compact convex sets
with non-empty interior. 
For a smooth convex body, we denote by $n_K(x)$ the unit outer normal to $\partial K$ at the point $x \in \partial K$.
The support function $h_K : {\mathbb R}^{m} \rightarrow {\mathbb R}$ associated with a convex body $K  \in {\mathcal K}^m$ is given by
 $h_K(u) = \sup \{ \langle x, u \rangle \, | \, x \in K \}$. 
The  phase space ${\mathbb R}^{2n}$ is equipped with the standard symplectic structure $\omega_n = dq \wedge dp$, and the standard linear complex structure $J : \R^{2n} \to \R^{2n}$. 
The symplectic action of a closed curve $\gamma$ in ${\mathbb R}^{2n}$ is defined as the integral $A(\gamma):= \int_{\gamma} \lambda$, where $\lambda$
 is a primitive of $\omega_n$.
Finally,  $W^{1,2}(S^1, {\mathbb R}^{2n})$ is the Banach space of absolutely continuous
$2\pi$-periodic functions whose derivatives belong to $L_2(S^1, {\mathbb R}^{2n})$.

\noindent {\bf Organization of the paper:}
In Section~\ref{sec-EHZ-of-p-products} we prove Theorem~\ref{thm-p-product-systolic}, Corollary~\ref{cor-asymp-Viterbo} and Proposition~\ref{prop-p-product-EHZ-capacity}. In Section~\ref{sec-GH-of-p-products} we prove Theorem~\ref{GH_p_formula} and Theorem~\ref{thm_inf_cap}. In Section~\ref{sec-p-decomposition} we prove Corollary~\ref{Cor-p-decomposition-of-the-ball-systolic} and Theorem~\ref{thm-application-of-conj-p-products}. The paper is concluded with an appendix including  a proof of Lemmas~\ref{formulas_coincide}. 

\noindent {\bf Acknowledgements:}
We wish to thank Viktor Ginzburg, Ba\c sak G\"urel, and Marco Mazzucchelli for 
helpful comments on higher-order symplectic capacities, and to Shiri Artstien-Avidan  for many stimulating discussions and useful remarks.  
The authors received funding from the European Research Council grant No. 637386. Y.O. is partially supported by the ISF grant No. 667/18.

\section{The EHZ-capacity of symplectic $p$-products} \label{sec-EHZ-of-p-products}

In this section we prove Proposition~\ref{prop-p-product-EHZ-capacity},  Theorem~\ref{thm-p-product-systolic}, and  Corollary~\ref{cor-asymp-Viterbo}.
We will give two separate proofs of Proposition~\ref{prop-p-product-EHZ-capacity}. The first is an analytic proof based on Clarke's dual action principle (see \cite{clarke}), and the second is a geometric proof based on an analysis of the dynamics of characteristics on $p$-products. 

We start with a few preliminary definitions and remarks. 
Consider a convex body $K$ in
${\mathbb R}^{2n}$ with a smooth boundary.
The restriction of
the symplectic form $\omega$ to the  hypersurface $\partial K$
canonically defines a one-dimensional sub-bundle, ${\rm ker} \, \omega_{|_{\partial K}}$, whose integral curves form the characteristic foliation
of $\partial K$. 
Recall that  $\ehzcap(K)$ is defined to be the minimal symplectic action among the closed characteristics on $\partial K$, or equivalently, the minimal period of a closed characteristic $\gamma$, where $\gamma$ is parametrized by $$\dot{\gamma}(t) = \frac{2 J n_K(\gamma(t))}{h_K(n_K(\gamma(t)))} .$$
We remark that although the above definition of  $\ehzcap(K)$  was given only for the class of convex bodies with smooth boundary, it
can be naturally generalized to the class of convex sets in ${\mathbb R}^{2n}$
with nonempty interior (see, e.g.,~\cite{AA-O1}). 
Moreover, since all the symplectic capacities considered in this paper are known to be continuous with respect to the Hausdorff metric on the class of convex domains, in what follows we can without loss of generality assume that all the convex domains are smooth and strictly convex. 

Next we describe the dynamics of a characteristic on the $p$-product $K \times_p T$, for two convex bodies $K$ and $T$.
Let $x \in \partial K$ and $y \in \partial T$. Let $(\alpha x, \beta y) \in \partial (K \times_p T)$ be a point on the boundary of the $p$-product of $K$ and $T$, i.e., such that $\alpha^p + \beta^p = 1$ and $\alpha, \beta \geq 0$.
A direct computation gives
\begin{equation} \label{eq_p_product_normal}
    \frac{ n_{K\times_p T} (\alpha x, \beta y)}{h_{K\times_p T} (n_{K\times_p T} (\alpha x, \beta y) )} = \Bigl (\alpha^{p-1} \frac{n_K(x)}{h_K(n_K(x))}, \beta^{p-1} \frac{n_T(y)}{h_T(n_T(y))} \Bigr). 
\end{equation} 
This equation shows that the two natural projections of the characteristic directions in $K \times_p T$ are the characteristic directions in $K$ and $T$, respectively.

The main ingredient in the first proof of 
Proposition~\ref{prop-p-product-EHZ-capacity} is the following formula  based on Clark's dual action principle: for a convex body $K \subset {\mathbb R}^{2n}$ and $p \geq 1$, one has
\begin{align}
c_{\rm EHZ}(K)^{\frac p 2} & = \pi^p \min_{z \in {\mathcal E}_n, \, A(z) = 1} \frac{1}{2\pi} \int_0^{2\pi} h_K^p(\dz(t)) dt  \label{eq_1-formula-ECH}  \\
& = \pi^p \min_{z \in {\mathcal E}_n, \; A(z) > 0} \frac{1}{2 \pi} A(z)^{-\frac{p}{2}} \int_0^{2 \pi} h_K^p(\dz(t)) dt,  \label{eq_2-formula-ECH} 
\end{align}
where ${\mathcal E}_n =  \{ z \in W^{1,2}(S^1, {\mathbb R}^{2n}) \, | \, \int_0^{2\pi} z(t)dt = 0  \}$, and $A(z)$ is the symplectic action of $z$.
Here~$(\ref{eq_1-formula-ECH})$ follows, e.g., from Proposition 2.1 in~\cite{AA-O2}, and~$(\ref{eq_2-formula-ECH})$ follows by rescaling. Moreover, following the proof of Proposition 2.1 in~\cite{AA-O2} the minimizer of \eqref{eq_1-formula-ECH} 
coincides, up to translation and rescaling, with a 
closed characteristic on the boundary $\partial K$.

For the proof of Proposition~\ref{prop-p-product-EHZ-capacity} we also need the following lemma, the proof of which is a simple exercise.
\begin{lemma}
	\label{Lemma-calculus}
	For $a,b>0$ one has
	$$ \min_{x \in [0,1]} ax^{q/2} + b(1-x)^{q/2} = 
	\begin{cases}
	\min \{a,b\}, & \text{if }1 \leq q \leq 2,\\
	\left(a^\frac{2}{2-q} + b^\frac{2}{2-q}\right)^{\frac{2-q}{2}},& \text{if} \ q > 2.
	\end{cases}$$
\end{lemma}

\begin{proof}[{\bf Proof of Proposition \ref{prop-p-product-EHZ-capacity}}]
	Let $K \subset \R^{2n}$ and $T\subset \R^{2m}$ be two convex bodies with smooth boundaries, and let $1 < p < \infty$. Note that the proof of the two cases $p=1$ and $p=\infty$ follows from the above case by a standard continuation argument. 
	Let $\gamma \subset \mathcal{E}_{n+m}$ be a minimizer of \eqref{eq_1-formula-ECH} for the body $K \times_p T$ and denote by $\gamma_1$ and $\gamma_2$ the projections of $\gamma$ on $\R^{2n}$ and $\R^{2m}$, respectively.
	From the fact that $\gamma$ is a rescaling and translation of a closed characteristic on $\partial ( K \times_p T)$, its velocity $\dot{\gamma}(t)$ is equal to a positive constant times $J n_{K\times_p T} (\gamma(t))$. Using \eqref{eq_p_product_normal} and the convexity of $K$ and $T$, we get that $A(\gamma_1),A(\gamma_2) > 0$.
	Moreover, the actions of $\gamma_1$ and $\gamma_2$ satisfy $A(\gamma_1) + A(\gamma_2) = 1$.
	Note moreover that from the definition of $K \times_p T$ it follows that
	$$ h_{K \times_p T}(x,y)^q = h_K(x)^q + h_T(y)^q, $$
	for $\frac{1}{p} + \frac{1}{q} = 1$. 
	Using \eqref{eq_1-formula-ECH} for $K\times_p T$ and \eqref{eq_2-formula-ECH} for $K$ and for $T$ gives
	 	\begin{align*}
	\ehzcap(K \times_p T) & = \left[ \pi^q \frac{1}{2 \pi} \int_0^{2 \pi} h_{K \times_p T}^q(\dg(t))dt \right]^{2/q} \\
	& = \left[ \pi^q \frac{1}{2 \pi} \int_0^{2 \pi} h_K^q(\dg_1(t))dt + \pi^q \frac{1}{2 \pi} \int_0^{2 \pi} h_T^q(\dg_2(t))dt \right]^{2/q} \\
	& \geq \left[ A(\gamma_1)^{q/2} \ehzcap(K)^{q/2} + A(\gamma_2)^{q/2} \ehzcap(T)^{q/2} \right]^{2/q} .
	\end{align*}
	For $p \geq 2$, using Lemma \ref{Lemma-calculus} one has that $$\ehzcap(K \times_p T) \geq \min \{ \ehzcap(K), \ehzcap(T) \}.$$
	Next, since for every symplectic subspace $E$ and any convex body $P$ one has $\ehzcap(P) \leq \ehzcap(\pi_E P)$, where $\pi_E$ stands for the projection operation, we get that $\ehzcap(K \times_p T) \leq \min \{\ehzcap(K),\ehzcap(T)\}$, and hence $$\ehzcap(K \times_p T) = \min \{\ehzcap(K),\ehzcap(T)\}.$$
		For $1 \leq p < 2$, Lemma \ref{Lemma-calculus} gives
	\begin{align*}
	\ehzcap(K \times_p T) & \geq \left[\left(\ehzcap(K)^\frac{q}{2-q} + \ehzcap(T)^\frac{q}{2-q}\right)^{\frac{2-q}{2}} \right]^{2/q} \\
	& = \left(\ehzcap(K)^\frac{1}{1-2/p} + \ehzcap(T)^\frac{1}{1-2/p}\right)^{1-2/p}.
	\end{align*}
		For the other direction, take the minimizers $\gamma_1 \in \mathcal{E}_n$ and $\gamma_2 \in \mathcal{E}_m$ of \eqref{eq_1-formula-ECH} for $K$ and $T$, respectively. Recall that they are normalized so that $A(\gamma_1) = A(\gamma_2) = 1$.
	Consider now the curve $\gamma = \left( \ehzcap(K)^{\frac{q}{2(2-q)}} \gamma_1, \ehzcap(T)^{\frac{q}{2(2-q)}} \gamma_2 \right)$. Then, 
	\begin{eqnarray*}
	\lefteqn{\ehzcap(K \times_p T) \leq} \\
	  &&A(\gamma)^{-1} \Bigg[ \pi^q \frac{1}{2 \pi} \int_0^{2 \pi} h_K^q(\ehzcap(K)^{\frac{q}{2(2-q)}} \dg_1(t)) dt \\
	  & & \qquad \qquad \qquad \qquad \qquad \qquad \qquad + \pi^q \frac{1}{2 \pi} \int_0^{2 \pi} h_T^q(\ehzcap(T)^{\frac{q}{2(2-q)}} \dg_2(t)) dt \Bigg]^{2/q} \\
	&  =& \left(\ehzcap(K)^{\frac{q}{2-q}} + \ehzcap(T)^{\frac{q}{2-q}} \right)^{-1} \Bigg[ \ehzcap(K)^{\frac{q^2}{2(2-q)}} \ehzcap(K)^{\frac{q}{2}} \\
	& & \qquad \qquad \qquad \qquad \qquad \qquad \qquad + \ehzcap(T)^{\frac{q^2}{2(2-q)}} \ehzcap(T)^{\frac{q}{2}} \Bigg]^{2/q} \\
	&  =& \left(\ehzcap(K)^{\frac{q}{2-q}} + \ehzcap(T)^{\frac{q}{2-q}} \right)^{2/q-1} 
	   = \left(\ehzcap(K)^{\frac{1}{1-2/p}} + \ehzcap(T)^{\frac{1}{1-2/p}} \right)^{1-2/p},
	\end{eqnarray*}
and the proof of the proposition is now complete. 
\end{proof}

Now let us examine closed characteristics on $K \times_p T$ from a dynamical point of view.
Note that \eqref{eq_p_product_normal} implies that if $\gamma_1(t) \subset \partial K$ and $\gamma_2(t) \subset \partial T$ are characteristics with 
$$\dot{\gamma_1}(t) = \frac{2 Jn_K(\gamma_1(t))}{h_K(n_K(\gamma_1(t)))} \ \ {\rm and} \  \ \dot{\gamma_2}(t) = \frac{2 Jn_T(\gamma_2(t))}{h_T(n_T(\gamma_2(t)))}$$ then $$\gamma(t) = (\alpha \gamma_1(\alpha^{p-2} t), \beta \gamma_2(\beta^{p-2} t)) $$ is a characteristic on $\partial (K \times_p T)$ with the parametrization $$\dot{\gamma}(t) = \frac{2 Jn_{K \times_p T}(\gamma(t))}{h_{K \times_p T}(n_{K \times_p T}(\gamma(t)))}.$$
Assume first that $\alpha \neq 0$ and $\beta \neq 0$. 
Note that the curve $\gamma$ is closed if and only if $\gamma_1$ and $\gamma_2$ are closed and there exists $t_0 \in \R$ such that $\gamma_1(\alpha^{p-2} t_0) = \gamma_1(0)$ and $\gamma_2(\beta^{p-2} t_0) = \gamma_2(0)$, hence $t_1 := \alpha^{p-2} t_0$ and $t_2 := \beta^{p-2} t_0$ are periods of $\gamma_1$ and $\gamma_2$, respectively, which satisfy $ t_1 / \alpha^{p-2} =  t_2 / \beta^{p-2}$. As $\alpha^p + \beta^p = 1$, one gets that 
\[\alpha^p = \frac{t_1^{\frac{p}{p-2}}}{t_1^{\frac{p}{p-2}} + t_2^{\frac{p}{p-2}}}, \]
and
\[ t_0=t_1 / \alpha^{p-2} = (t_1^{\frac{p}{p-2}} + t_2^{\frac{p}{p-2}})^\frac{p-2}{p}. \]
In the case that either $\alpha=0$ or $\beta=0$ one gets that $\gamma = (0, \gamma_2)$ or $\gamma = (\gamma_1, 0)$ and $t_0=t_2$ or $t_0=t_1$, respectively.
Finally, recall that in the above parametrization the minimal period gives the capacity, and hence 
\begin{align*}
    \ehzcap(K\times_p T) &= \min\{t_0 : t_0 \in \mathcal{P}(K \times_p T) \} \\
    &= \min\{ \min((t_1^{\frac{p}{p-2}} + t_2^{\frac{p}{p-2}})^\frac{p-2}{p}, t_1, t_2): t_1 \in \mathcal{P}(K), \; t_2 \in \mathcal{P}(T) \} \\
    &= \min((\ehzcap(K)^{\frac{p}{p-2}} + \ehzcap(T)^{\frac{p}{p-2}})^\frac{p-2}{p}, \ehzcap(K), \ehzcap(T)),
\end{align*}
where $\mathcal{P}(Q)$ is the set of periods of all the closed characteristics on $\partial Q$ for a convex body $Q \subset \R^{2n}$.
Note that if $p > 2$, then $$\forall x,y>0, \; (x^{\frac{p}{p-2}} + y^{\frac{p}{p-2}})^\frac{p-2}{p} > \max\{x, y\},$$ and if $1 \leq p < 2$, then $$\forall x,y>0, \; (x^{\frac{p}{p-2}} + y^{\frac{p}{p-2}})^\frac{p-2}{p} < \min\{ x, y \}.$$
Hence when $p >2$,
\[ \ehzcap(K\times_p T) = \min\{\ehzcap(K), \ehzcap(T) \}, \]
and when $1 \leq p < 2$,
\begin{align*}  
\ehzcap(K\times_p T) &= (\ehzcap(K)^{\frac{p}{p-2}} + \ehzcap(T)^{\frac{p}{p-2}})^\frac{p-2}{p}. 
\end{align*}
This observation reproves Proposition~\ref{prop-p-product-EHZ-capacity} in a somewhat  different way.

For the proof of Theorem \ref{thm-p-product-systolic} we need the following lemma, which is a simple corollary of the well-known formula $\vol(K) = \frac{1}{\Gamma(1+\frac{n}{p})} \int_{\R^{n}} e^{-\|x\|_K^p}dx$ for the volume of a convex body $K \subset \R^n$ (see e.g. \cite{schneider}).
\begin{lemma} \label{lem_vol}
	For convex $K \subset \R^{n}$ and $T \subset \R^m$ one has
	$$  \vol(K \times_p T) = \frac{\Gamma(\frac{n}{p}+1) \Gamma(\frac{m}{p}+1)}{\Gamma(\frac{m+n}{p}+1)} \vol(K) \vol(T) .$$
\end{lemma}

\begin{proof}[{\bf Proof of Theorem \ref{thm-p-product-systolic}}]
	
	First let $p \geq 2.$ Then, from Lemma~\ref{lem_vol} and the fact that $K \times_2 T \subseteq K \times_pT$ it follows that
	$$ \vol(K\times_p T) \geq \vol(K \times_2 T) = \frac{m! n!}{(m+n)!} \vol(K) \vol(T),$$
	with equality only when $p=2$. Next, 
	from  Proposition~\ref{prop-p-product-EHZ-capacity}  it follows that
	$$\ehzcap(K \times_p T) = \min\{\ehzcap(K), \ehzcap(T)\} \leq \ehzcap(K)^\frac{n}{m+n} \ehzcap(T)^\frac{m}{m+n}, $$
	with equality only when $\ehzcap(K) = \ehzcap(T)$.
	Hence
	$$ sys_{m+n}(K \times_p T)^{m+n} \leq \frac{\ehzcap(K)^n}{n! \vol(K)} \frac{\ehzcap(T)^m}{m! \vol(T)} = sys_n(K)^n sys_m(T)^m,$$
	with equality only when $p=2$ and when $\ehzcap(K) = \ehzcap(T)$.

	Next, let $1 \leq p < 2.$
	Using Proposition~\ref{prop-p-product-EHZ-capacity} and the
	following inequality, which is a simple consequence of the inequality of the weighted arithmetic and geometric means, one has:
	\begin{align} \label{eq-am-gm}
	     \ehzcap(K\times_p T) &= \left( \ehzcap(K)^{\frac{p}{p-2}} + \ehzcap(T)^{\frac{p}{p-2}} \right)^\frac{p-2}{p} \\
	     &\leq \left( \frac{n+m}{m^\frac{m}{n+m} n^\frac{n}{n+m}} \right)^\frac{p-2}{p} \ehzcap(K)^\frac{n}{n+m} \ehzcap(T)^\frac{m}{n+m} \nonumber
	\end{align}
Plugging \eqref{eq-am-gm} in the formula for the systolic ratio gives
	\begin{align*}
	sys_{m+n}(K \times_p T)^{m+n} &= \frac{\left( \ehzcap(K)^{\frac{p}{p-2}} + \ehzcap(T)^{\frac{p}{p-2}} \right)^{\frac{p-2}{p} (m+n)}}{(m+n)! \vol(K \times_p T)} \\
	&\leq \frac{ \left( \frac{(n+m)^{m+n}}{m^m n^n} \right)^\frac{p-2}{p} \ehzcap(K)^n \ehzcap(T)^m}{(n+m)! \frac{\Gamma(1+2n/p) \Gamma(1+2m/p)}{\Gamma(1+(2n+2m)/p)} \vol(K) \vol(T)} \\
	&= g \left (\frac{1}{p} \right) \frac{\ehzcap(K)^n}{n! \vol(K)} \frac{\ehzcap(T)^m}{m! \vol(T)},
	\end{align*}
	where
	$$ g(x) := \left( \frac{(n+m)^{m+n}}{m^m n^n} \right)^{1-2x} \frac{\Gamma(1+(2n+2m)x)}{\Gamma(1+2nx)\Gamma(1+2mx)} \frac{n!m!}{(m+n)!}. $$
	Hence, to finish the proof it remains to check that 
	$$ g(x) \leq 1 ,$$
	for every $\frac{1}{2} < x \leq 1$.
	As in \cite[Lemma 4.5]{matthias_henze}, a direct computation shows that $$\frac{d^2}{dx^2} \ln[ g(x) ] > 0,$$
	and hence  $$g(x) \leq \max \{ g(1/2), g(1) \},$$ for every $x \in [1/2,1]$.
	As $g(1/2) = 1$ it remains to show that 
	$$ 1 \geq g(1) = \frac{m^m n^n}{(n+m)^{n+m}} \frac{(2n+2m)!}{(2n)!(2m)!} \frac{n! m!}{(n+m)!} .$$
	One is able to show this directly using, e.g., the following bounds:
	$$ \sqrt{2\pi} n^{n+1/2} e^{-n} e^{\frac{1}{12n+1}} < n! < \sqrt{2\pi} n^{n+1/2} e^{-n} e^{\frac{1}{12n}} .$$
	Alternatively, note that $g(1)$ is the systolic ratio of a 1-product of Euclidean balls of different radii. Indeed, choose $K \subset \R^{2n}$ to be the Euclidean ball of capacity 1, and choose $T \subset \R^{2m}$ to be the Euclidean ball of capacity $c(T) = \frac{n}{m} c(K).$ Now one has equality in \eqref{eq-am-gm} above when $p=1,$ and hence 
	$$ sys_{m+n}(K \times_1 T)^{m+n} = g(1) sys_n(K)^n sys_m(T)^m = g(1).$$
	Since $K \times_1 T$ is $S^1$-symmetric, $\ehzcap(K \times_1 T) = c_{\rm Gr}(K \times_1 T)$ and the largest symplectic ball inside $K \times_1 T$ is a Euclidean ball (cf. Proposition 1.4 in~\cite{Gu-Ha-Ra}). This gives $g(1) = sys_{m+n}(K \times_1 T)^{m+n} < 1$, as required.
	Moreover, $\frac{d^2}{dx^2} \ln[ g(x) ] > 0$ implies that $g(x) < g(1/2)$ for $x \in (1/2,1]$. Therefore, the inequality is strict whenever $1 \leq p<2.$
	\end{proof}

\begin{proof}[{\bf Proof of Corollary~\ref{cor-asymp-Viterbo}}]
	Assume without loss of generality that $m \geq n$. For a convex $K \subset \R^{2n}$ take a ball $B^{2(m-n)}[\ehzcap(K)]$ with the same capacity as $K$. 
	Now Theorem~\ref{thm-p-product-systolic} gives
	\begin{align*} sys_n(K)^n & = sys_n(K)^n sys_{m-n}(B^{2(m-n)}[\ehzcap(K)])^{m-n} \\
	& = sys_m(K \times_2 B^{2(m-n)}[\ehzcap(K)])^m.
	\end{align*}
	Thus, Conjecture \ref{conj-viterbo} in dimension $2m$ implies Conjecture \ref{conj-viterbo} in dimension $2n$ for any $n \leq m$.
	Moreover, if we know that $sys_n(K) \leq \alpha(n)$ for any $K$ and $n$, for some function $\alpha$ that satisfies $\alpha(n) \xrightarrow{n \to \infty} 1$, then we can take a product of $K$ with itself $m$-times to conclude that
	$$ sys_{nm}(K \times_2 K \times_2 \cdots \times_2 K)^{mn} = (sys_n(K)^n)^m, $$
	and since $sys_{nm}(K \times_2 \cdots \times_2 K) \leq \alpha(nm)$, letting $m \to \infty$ we get $$sys_n(K) \leq 1.$$
	This completes the proof of the corollary. 
\end{proof}

\section{Capacities of $p$-products of toric domains} \label{sec-GH-of-p-products}
In this section we prove Theorem \ref{GH_p_formula}~ and Theorem~\ref{thm_inf_cap}.
Define the moment map $\mu : \C^n \to \R^n_+$ by $\mu(z_1,\ldots,z_n) = \pi(|z_1|^2,\ldots,|z_n|^2)$.
For a domain $\Omega \subset \R^{n}_+$ define the corresponding toric domain by $X_\Omega = \mu^{-1}(\Omega) \subset \C^n$ and the continuation of $\Omega$ to an unconditional domain by $\widehat{\Omega} = \{ (x_1,\ldots,x_n) \in \R^n | (|x_1|,\ldots,|x_n|) \in \Omega \}$.
We say $X_\Omega$ is a convex (concave) toric domain if $\widehat{\Omega}$ is a convex (concave) set.
For more information on toric domains see, e.g., Section 2 of~\cite{Gu-Ha-Ra}. For the proof of the theorem mentioned above we shall need the following two simple lemmas.

\begin{lemma} \label{lem-1-about-toric-p-product} Let $X_{\Omega_1} \subset {\mathbb C}^{n}$ and $X_{\Omega_2} \subset {\mathbb C}^m$ be two toric domains. Then, 
for every $1 \leq p \leq \infty$ one has 
	$X_{\Omega_1} \times_p X_{\Omega_2} = X_{\Omega_1 \times_{p/2} \Omega_2} $.
\end{lemma}
\begin{proof} [{\bf Proof of Lemma~\ref{lem-1-about-toric-p-product}}]
	First note that $\|\mu(x)\|_\Omega = \|x\|_{X_\Omega}^2$. Indeed, since $\mu(\alpha x) = \alpha^2 \mu(x)$, one has $\frac{\mu(x)}{\lambda^2} \in \Omega \iff \frac{x}{\lambda} \in X_\Omega$ for $\lambda \in \R_+$.
	Next, for $(x,y) \in \C^n \times \C^m$, 
	\begin{align*}
	(x,y) \in X_{\Omega_1 \times_{p/2} \Omega_2} & \iff (\mu(x),\mu(y)) \in \Omega_1 \times_{p/2} \Omega_2 \\
	& \iff \|\mu(x)\|_{\Omega_1}^{p/2} + \|\mu(y)\|_{\Omega_2}^{p/2} \leq 1 \\ 
	& \iff \|x\|^p_{X_{\Omega_1}} + \|y\|^p_{X_{\Omega_2}} \leq 1 \\
	& \iff (x,y) \in X_{\Omega_1} \times_p X_{\Omega_2},
	\end{align*}
which completes the proof. 
\end{proof}

\begin{lemma} \label{lem-2-about-toric-p-product}
	For convex domains $\Omega_1 \subset \R^n$ and $\Omega_2 \subset {\mathbb R}^m$, the domain $\Omega_1 \times_p \Omega_2$ is convex for $p\geq 1$, and for concave domains $\Omega_1 \subset {\mathbb R}^n$ and $\Omega_2 \subset \R^m$, the domain $\Omega_1 \times_p \Omega_2$ is concave for $0 < p\leq 1$.
\end{lemma}
\begin{proof} [{\bf Proof of Lemma~\ref{lem-2-about-toric-p-product}}]	For $(x,y),(w,z) \in {\Omega_1} \times_p {\Omega_2}$ and $ p \geq 1$, it follows from Minkowski inequality that
	\begin{align*}
	\|(x,y) + (w,z)\|_{\Omega_1 \times_p \Omega_2} & = \left( \|x+w\|_{\Omega_1}^p + \|y+z\|_{\Omega_2}^p \right)^{1/p} \\
	& \leq \left( \left( \|x\|_{\Omega_1} + \|w\|_{\Omega_1} \right)^p + \left( \|y\|_{\Omega_2} + \|z\|_{\Omega_2} \right)^p \right)^{1/p}\\
	& \leq \left( \|x\|_{\Omega_1}^p + \|y\|_{\Omega_2}^p \right)^{1/p} + \left( \|w\|_{\Omega_1}^p + \|z\|_{\Omega_2}^p \right)^{1/p} \\
	& = \|(x,y)\|_{{\Omega_1} \times_p {\Omega_2}} + \|(w,z)\|_{{\Omega_1} \times_p {\Omega_2}},
	\end{align*}
	which proves the first part of the lemma. In a similar manner, a direct computation shows that for $0 < p \leq 1$, using now the reverse Minkowski inequality,  one has
	$$	\|(x,y) + (w,z)\|_{{\Omega_1} \times_p {\Omega_2}} \geq  \|(x,y)\|_{{\Omega_1} \times_p {\Omega_2}} + \|(w,z)\|_{{\Omega_1} \times_p {\Omega_2}} ,$$
	which proves the second part of the lemma. 
\end{proof}

Next, consider the family $\{c^k_{\rm GH} \}_{k=1}^{\infty}$ of the Gutt-Hutchings symplectic capacities introduced in~\cite{Gu-Ha}. The proofs of Theorem~\ref{GH_p_formula} 
is based on the followin:
\begin{theorem}[Gutt, Hutchings \cite{Gu-Ha}] \label{GH-thm-toric-domains}
	For a convex toric domain $X_\Omega \subset {\mathbb C}^n$,
	$$c^k_{\rm GH}(X_\Omega) = \min \left\{ h_\Omega(v) \bigg| v = (v_1,\ldots,v_n) \in \N^n, \, \sum_{i=1}^n v_i = k \right\}, $$
	and for a concave toric domain $X_\Omega \subset {\mathbb C}^n$,
	$$c^k_{\rm GH}(X_\Omega) = \max\left\{[v]_\Omega \;\bigg|\; v\in\N^n_{>0},\;\sum_iv_i = k+n-1\right\}, $$
	where $ [v]_\Omega = \min\left\{\langle v,w\rangle \;\big|\; w\in\Sigma\right\} $, and $\Sigma$ is the clouser of the set $\partial \Omega \cap {\mathbb R}^n_{>0}$.
\end{theorem}

\begin{proof} [{\bf Proof of Theorem \ref{GH_p_formula}}]
Consider first the case of convex toric domains $K$ and $T$, and $p > 2$. As before, the case $p=2$ will follow from a standard continuation argument. Then, using Lemma~\ref{lem-1-about-toric-p-product} and Lemma~\ref{lem-2-about-toric-p-product}, one has that
	$K \times_p T = X_{\mu(K) \times_{p/2} \mu(T)}$ is a convex toric domain, and hence, by Theorem~\ref{GH-thm-toric-domains},
	\begin{eqnarray*}
	\lefteqn{c^k_{\rm GH}(K \times_p T) = c^k_{\rm GH}(X_{\mu(K) \times_{p/2} \mu(T)})} \\
	&= & \min \left\{ h_{\mu(K) \times_{p/2} \mu(T)}(v) \bigg| v \in \N^{n+m}, \sum_{i=1}^{n+m} v_i = k \right\} \\
	&= & \min \Bigg\{ \left( h_{\mu(K)}^\frac{p}{p-2}(v_1,\ldots,v_n) +  h_{\mu(T)}^\frac{p}{p-2}(v_{n+1},\ldots,v_{n+m}) \right)^\frac{p-2}{p} \\
	& & \qquad \qquad \qquad \bigg| v \in \N^{n+m}, \sum_{i=1}^{n+m} v_i = k \Bigg\} \\
	&= & \min_{i+j = k} \Bigg[ \min \left\{ h_{\mu(K)}^\frac{p}{p-2}(v) \bigg| v\in\N^n, \sum_{r=1}^n v_r = i \right\} \\
	& & \qquad \qquad \qquad + \min \left\{ h_{\mu(T)}^\frac{p}{p-2}(v) \bigg| v\in\N^m, \sum_{r=1}^m v_r = j \right\} \Bigg]^\frac{p-2}{p} \\
	&= & \min_{i+j = k} \left[ c^i_{\rm GH}(K)^\frac{p}{p-2} + c^j_{\rm GH}(T)^\frac{p}{p-2} \right]^\frac{p-2}{p}.
	\end{eqnarray*}
For the case of concave domains $K$ and $T$ and $1 \leq p \leq 2$ one has the following.
	$$ [(x,y)]_{\mu(K) \times_{p/2} \mu(T)} = \left( [x]_{\mu(K)}^\frac{p}{p-2} + [y]_{\mu(T)}^\frac{p}{p-2} \right)^\frac{p-2}{p}, $$
	and hence, using similar arguments as above, one has
	\begin{eqnarray*}
	\lefteqn{c^k_{\rm GH}(K \times_p T) = c^k_{\rm GH}(X_{\mu(K) \times_{p/2} \mu(T)})}\\
	&= & \max \Bigg\{ [v]_{\mu(K) \times_{p/2} \mu(T)} \bigg| v \in \N^{n+m}_{>0}, \sum_{i=1}^{n+m} v_i = k + m + n -1 \Bigg\} \\
	&= & \max \Bigg\{ \left( [(v_1,\ldots,v_n)]_{\mu(K)}^\frac{p}{p-2} +  [(v_{n+1},\ldots,v_{n+m})]_{\mu(T)}^\frac{p}{p-2} \right)^\frac{p-2}{p} \\
	& & \qquad \qquad \qquad \bigg| v \in \N^{n+m}_{>0}, \sum_{i=1}^{n+m} v_i = k+m+n-1 \Bigg\} \\
	&= & \max_{\substack{i+j = k+1 \\i>0,j>0}} \Bigg[ \min \left\{ [v]_{\mu(K)}^\frac{p}{p-2} \bigg| v\in\N^n_{>0}, \sum_{r=1}^n v_r = i+n-1 \right\} \\
	& &  \qquad \qquad \qquad + \min \left\{ [v]_{\mu(T)}^\frac{p}{p-2} \bigg| v\in\N^m_{>0}, \sum_{r=1}^m v_r = j+m-1 \right\} \Bigg]^\frac{p-2}{p} \\
	&= & \max_{\substack{i+j = k+1 \\i>0,j>0}} \left[ c^i_{\rm GH}(K)^\frac{p}{p-2} + c^j_{\rm GH}(T)^\frac{p}{p-2} \right]^\frac{p-2}{p},
	\end{eqnarray*}
and the proof of the theorem is complete. 
\end{proof}

\begin{proof}[{\bf Proof of Theorem \ref{thm_inf_cap}}]
 Note that it is enough to prove \eqref{eq-c-inf} for the product of two convex domains. Namely, let $K \subset \R^{2n}$ and $T \subset {\mathbb R}^{2m}$ be convex domains such that the limits $\frac{c^k_{\rm EH}(K)}{k}$ and $\frac{c^k_{\rm EH}(T)}{k}$ exist. Then we wish to prove that 
 $$ c_\infty(K \times_p T) = \left( c_\infty(K)^{\frac{-p}{2}} + c_\infty(T)^{\frac{-p}{2}} \right)^{\frac{-p}{2}} .$$
 We will prove this in the case where $p>2$. The proof for $1 \leq p < 2$ is similar. One can derive the case $p=2$ by continuity of the capacities and the inclusions 
 $$K \times_{2-\varepsilon} T \subset K \times_2 T \subset K \times_{2+\varepsilon} T.$$
 First we will show that
 $$ c_\infty(K \times_p T) \leq \left( c_\infty(K)^{\frac{-p}{2}} + c_\infty(T)^{\frac{-p}{2}} \right)^{\frac{-p}{2}} .$$
 Denote 
 $$w = \frac{c_\infty(T)^\frac{p}{2}}{c_\infty(K)^\frac{p}{2}+c_\infty(T)^\frac{p}{2}}$$
 and choose 
 $$i_k = \ceil{k w}, \qquad j_k = k - i_k = \floor{k(1-w)}.$$
 By the definition of $c_\infty$, for every $\varepsilon>0$ there exists a large enough $k$ such that 
 $$ c^{i_k}_{\rm EH}(K) < i_k(c_\infty(K)+\varepsilon), \quad c^{j_k}_{\rm EH}(T) < j_k(c_\infty(T)+\varepsilon), \quad \frac{1}{k} < \varepsilon. $$
Using Conjecture \ref{Conj-EH-cap-of-p-products} we get
\begin{align*}
c^k_{\rm EH}(K \times_p T) &\leq \left[ c^{i_k}_{\rm EH}(K)^\frac{p}{p-2} + c^{j_k}_{\rm EH}(T)^\frac{p}{p-2} \right]^\frac{p-2}{p}\\
& < \left[ i_k^\frac{p}{p-2}(c_\infty(K)+\varepsilon)^\frac{p}{p-2} + j_k^\frac{p}{p-2}(c_\infty(T)+\varepsilon)^\frac{p}{p-2} \right]^\frac{p-2}{p} \\
& \leq \left[ k^\frac{p}{p-2} (w + \frac{1-\{wk\}}{k})^\frac{p}{p-2}(c_\infty(K)+\varepsilon)^\frac{p}{p-2} + k^\frac{p}{p-2} (1-w)^\frac{p}{p-2}(c_\infty(T)+\varepsilon)^\frac{p}{p-2} \right]^\frac{p-2}{p} \\
& \leq k \left( \left[ w^\frac{p}{p-2} c_\infty(K)^\frac{p}{p-2} + (1-w)^\frac{p}{p-2} c_\infty(T)^\frac{p}{p-2}   \right]^\frac{p-2}{p} + O(\varepsilon) \right),
\end{align*}
By substituting $w$ one gets
\begin{align} \label{eq-upper-bound-c-inf}
\nonumber \frac{c^k_{\rm EH}(K \times_p T)}{k} &< \frac{c_\infty(K) c_\infty(T)}{c_\infty(K)^\frac{p}{2}+c_\infty(T)^\frac{p}{2}} \left(c_\infty(T)^\frac{p}{2}+c_\infty(K)^\frac{p}{2}\right)^\frac{p-2}{p} + O(\varepsilon) \\
	&= \left(c_\infty(K)^{-\frac{p}{2}} + c_\infty(T)^{-\frac{p}{2}}\right)^{-\frac{2}{p}} + O(\varepsilon).
\end{align}

	For the other direction, denote by $\hat{i}_k, \hat{j}_k$ the minimizers in the formula for $c^k_{\rm EH}(K \times_p T)$ from Conjecture \ref{Conj-EH-cap-of-p-products}. 
	We consider two cases. The first is that both $\hat{i}_k$ and $\hat{j}_k$ goes to infinity as $k$ goes to infinity, and the second is that one of the indices, say $\hat{i}_k$, is bounded. In the first case, 
	since $\hat{i}_k \to \infty$ and $\hat{j}_k \to \infty$ as $k \to \infty$, the definition of $c_\infty$ implies that for every $\varepsilon>0$ there exists a large enough $k$ such that
	$$ c^{\hat{i}_k}_{\rm EH}(K) > \hat{i}_k(c_\infty(K)-\varepsilon), \quad c^{\hat{j}_k}_{\rm EH}(T) > \hat{j}_k(c_\infty(T)-\varepsilon). $$
	Hence, one has
	\begin{align*}
	c^k_{\rm EH}(K \times_p T) &= \left[ c^{\hat{i}_k}_{\rm EH}(K)^\frac{p}{p-2} + c^{\hat{j}_k}_{\rm EH}(T)^\frac{p}{p-2} \right]^\frac{p-2}{p} \\
	& > \left[  \hat{i}_k^\frac{p}{p-2}(c_\infty(K)-\varepsilon)^\frac{p}{p-2} + \hat{j}_k^\frac{p}{p-2}(c_\infty(T)-\varepsilon)^\frac{p}{p-2} \right]^\frac{p-2}{p}.
	\end{align*} 
	Without the restriction that $\hat{i}_k,\hat{j}_k \in \N$, the right hand side is minimized when 
	$$ \hat{i}_k:= \frac{k (c_\infty(T)-\varepsilon)^\frac{p}{2}}{(c_\infty(K)-\varepsilon)^\frac{p}{2}+(c_\infty(T)-\varepsilon)^\frac{p}{2}} , \quad \hat{j}_k:= \frac{k (c_\infty(K)-\varepsilon)^\frac{p}{2}}{(c_\infty(K)-\varepsilon)^\frac{p}{2}+(c_\infty(T)-\varepsilon)^\frac{p}{2}} , $$
	and hence,
	\begin{align*}
	c^k_{\rm EH}(K \times_p T) &> k\left((c_\infty(K)-\varepsilon)^{-\frac{p}{2}} + (c_\infty(T)-\varepsilon)^{-\frac{p}{2}}\right)^{-\frac{2}{p}}.
	\end{align*}
	Consequently,
	$$ 	\frac{c^k_{\rm EH}(K \times_p T)}{k} > \left(c_\infty(K)^{-\frac{p}{2}} + c_\infty(T)^{-\frac{p}{2}}\right)^{-\frac{2}{p}} + O(\varepsilon) .$$
	Suppose now that $\hat{i}_k\leq i_0$, which means $\hat{j}_k \geq k - i_0$, for every $k$. For every $\varepsilon>0$, take a large enough $k$ so that
	\begin{align*}
	c^k_{\rm EH}(K \times_p T) = \left[ c^{\hat{i}_k}_{\rm EH}(K)^\frac{p}{p-2} + c^{\hat{j}_k}_{\rm EH}(T)^\frac{p}{p-2} \right]^\frac{p-2}{p} > \left[ c_{1}(K)^\frac{p}{p-2} + \left((k-i_0)\left(c_\infty(T)-\varepsilon\right)\right)^\frac{p}{p-2} \right]^\frac{p-2}{p}.
	\end{align*}
	Hence
	$$c_\infty(K \times_p T) = \lim_{k \to \infty} \frac{c^k_{\rm EH}(K \times_p T)}{k} > c_\infty(T) $$
	This contradicts \eqref{eq-upper-bound-c-inf}, since $c_\infty(T) > \left(c_\infty(K)^{-\frac{p}{2}} + c_\infty(T)^{-\frac{p}{2}}\right)^{-\frac{2}{p}} .$ The proof of the theorem is now complete. 
\end{proof}

\section{On the $p$-decomposition of a symplectic ball} \label{sec-p-decomposition}

In this section we prove Corollary~\ref{Cor-p-decomposition-of-the-ball-systolic} and Theorem \ref{thm-application-of-conj-p-products}. We start with the proof of the former.  

\begin{proof}[{\bf Proof of Corollary~\ref{Cor-p-decomposition-of-the-ball-systolic}}] Assume that Conjecture \ref{conj-viterbo} holds. 
The first part of the corollary follows immediately from the fact that if $p \neq 2$ then 
\[ sys_{m+n} (K \times_p T)^{m+n} < sys_n(K)^n sys_m(T)^m \leq 1, \]
so the systolic ratio is strictly smaller than that of the ball. Similarly, if one has  $\widetilde{B} = X \times_2 Y$ and $sys_n(X) < 1$ or $sys_m(Y) < 1$ then
\begin{equation*} sys_{m+n} (\widetilde{B})^{m+n} \leq sys_n(K)^n sys_m(T)^m < 1, 
\end{equation*} which completes the proof of the corollary.  \end{proof}

The following lemma, whose proof is a direct consequence of Theorem 1.2 and Lemma 3.4 in \cite{Gi-Gu}, is key in the proof of Theorem \ref{thm-application-of-conj-p-products}.

\begin{lemma} \label{lem-cap-n-inequality}
    For any convex $K \subset \R^{2n}$, and any $i \in \N$, 
    $$ c^i_{\rm EH}(K) < c^{n+i}_{\rm EH}(K) .$$
\end{lemma}

\begin{proof}[{\bf Proof of Theorem \ref{thm-application-of-conj-p-products}}]
	Assume without loss of generality that $r=1$ and $n \leq m$.
	Assume, by contradiction, that $B^{2n+2m}[1] = X \times_p Y$.
	First consider the case $p > 2$. From Conjecture \ref{Conj-EH-cap-of-p-products} one has 
	$$ c^k_{\rm EH}(X \times_p Y) = \min_{i+j = k} \left[ c^i_{ \rm EH}(X)^{\frac{p}{p-2}} + c^j_{\rm EH}(Y)^{\frac{p}{p-2}} \right]^\frac{p-2}{p} \leq \min \{c^k_{\rm EH}(X), c^k_{\rm EH}(Y)\},$$
	and hence $c^k_{ \rm EH}(X) \geq 1$, $c^k_{\rm EH}(Y) \geq 1$ for any $k \leq m+n$.
	Consider $k = m + 1$. Using Conjecture \ref{Conj-EH-cap-of-p-products} again, we get that
	$$ 1 = c^{m+1}_{\rm EH}(X \times_p Y) = \min_{i+j = m+1} \left[ c^i_{\rm EH}(X)^{\frac{p}{p-2}} + c^j_{\rm EH}(Y)^{\frac{p}{p-2}} \right]^\frac{p-2}{p}. $$
	Note that for $0 < i,j \leq m$ one has
	$$ \left[ c^i_{\rm EH}(X)^{\frac{p}{p-2}} + c^j_{\rm EH}(Y)^{\frac{p}{p-2}} \right]^\frac{p-2}{p} > c^i_{\rm EH}(X) \geq 1, $$
	and using Lemma \ref{lem-cap-n-inequality} for $i = m+1, \, j=0$, and $i=0, \, j = m+1$, one has
	$$ c^{m+1}_{\rm EH}(X) > c^1_{\rm EH}(X) \geq 1,\quad c^{m+1}_{\rm EH}(Y) > c^1_{\rm EH}(Y) \geq 1, $$
	which is a contradiction. Next, 
	for $1 \leq p < 2$, assume that $i,j$ are maximizers in the formula from Conjecture \ref{Conj-EH-cap-of-p-products} for the index $k = m+n$, i.e.,
	$$ 1 = c^{m+n}_{\rm EH}(X \times_p Y) = \left[ c^i_{\rm EH}(X)^{\frac{p}{p-2}} + c^j_{\rm EH}(Y)^{\frac{p}{p-2}} \right]^\frac{p-2}{p}. $$
	Since $i+j = m+n+1$ we may assume without loss of generality that $i > n$, and hence Lemma \ref{lem-cap-n-inequality} gives $c^i_{\rm EH}(X) > c^1_{\rm EH}(X)$. In addition note that
	$$ 1 = c^1_{\rm EH}(X \times_p Y) = \left[ c^1_{\rm EH}(X)^{\frac{p}{p-2}} + c^1_{\rm EH}(Y)^{\frac{p}{p-2}} \right]^\frac{p-2}{p}. $$
	Thus, we conclude that
	\begin{align*}
	 1 &= \left[ c^i_{\rm EH}(X)^{\frac{p}{p-2}} + c^j_{\rm EH}(Y)^{\frac{p}{p-2}} \right]^\frac{p-2}{p} > \left[ c^1_{\rm EH}(X)^{\frac{p}{p-2}} + c^j_{\rm EH}(Y)^{\frac{p}{p-2}} \right]^\frac{p-2}{p} \\
	 &\geq \left[ c^1_{\rm EH}(X)^{\frac{p}{p-2}} + c^1_{\rm EH}(Y)^{\frac{p}{p-2}} \right]^\frac{p-2}{p} = 1,
	\end{align*}
	which is again a contradiction. This completes the first part of the theorem. 
	
	Finally, when $p=2$, note that $c^k_{\rm EH}(X\times_2 Y)$ is the $k^{\rm th}$ term in the sequence  $\{c^i_{\rm EH}(X), c^j_{\rm EH}(Y)\}$ arranged in non-decreasing order with repetitions. Denote the elements in this sequence by $M_k$. 
	From the fact that $c^k_{\rm EH}(X\times_2 Y) = 1$ for $1\leq k \leq m+n$, one has that $M_1=M_2=\ldots=M_{m+n}=1$. By means of Lemma \ref{lem-cap-n-inequality} we get that the only way this is possible is if 
	$$c^1_{\rm EH}(X)=\ldots=c^n_{\rm EH}(X)=1, \ \ {\rm and} \ \ c^1_{\rm EH}(Y) = \ldots = c^m_{\rm EH}(Y)=1.$$ Next, using the fact that $M_{m+n+1}=\ldots=M_{2m+2n}=2$ we get again by invoking Lemma \ref{lem-cap-n-inequality} that $$c^{n+1}_{\rm EH}(X)=\ldots=c^{2n}_{\rm EH}(X)=2 \ \ {\rm  and} \ \ c^{m+1}_{\rm EH}(Y) = \ldots = c^{2m}_{\rm EH}(Y)=2.$$
	Continuing this argument by induction completes the proof.
\end{proof}

\section{Appendix}
Denote by $M_k(K,T)$ the $k$-th term in the sequence $A(K,T) := \{c^i_{\rm EH}(K),c^j_{\rm EH}(T)\}$, arranged in non-decreasing order with repetitions.
\begin{lemma} \label{combinatorial-lemma}
	\label{formulas_coincide}
	The two expressions in Conjecture \ref{Conj-EH-cap-of-p-products} 
	coincide for $p=2$, and moreover they are equal to $M_k(K,T)$, i.e.
	$$ \min_{i+j=k} \max \{c^i_{\rm EH}(K),c^j_{\rm EH}(T)\} = \max_{i+j=k+1} \min \{c^i_{\rm EH}(K),c^j_{\rm EH}(T)\} = M_k(K,T). $$
\end{lemma}

\begin{proof}[{\bf Proof of Lemma~\ref{combinatorial-lemma}}]
	
	First, we may assume that all elements in $A(K,T)$ are unique. Otherwise, one can assign an arbitrary strict total order $\prec$ to $A(K,T)$ that satisfies that if $a,b \in A(K,T)$ and $a<b$ then $a \prec b$, and $c^i_{\rm EH}(K) \prec c^j_{\rm EH}(K)$ and $c^i_{\rm EH}(T) \prec c^j_{\rm EH}(T)$ for all $i < j$. Then one can continue the proof where each time we consider inequalities between elements of $A(K,T)$ one can consider instead the same inequalities with $\prec$.
	
	We start with the case $\min_{i+j=k} \max \{c^i_{\rm EH}(K),c^j_{\rm EH}(T)\} \geq M_k(K,T).$
	Let $i,j$ be the minimizers in the above inequality, and assume without loss of generality that $c^i_{\rm EH}(K) \geq c^j_{\rm EH}(T)$.
	Then the maximum is $c^i_{\rm EH}(K)$, and there are at least $k$ elements in $\{c^l_{\rm EH}(K)\}_{l=1}^\infty \cup \{c^l_{\rm EH}(T) \}_{l=1}^\infty$ that are smaller than $c^i_{\rm EH}(K)$. Indeed, $c^1_{\rm EH}(K), \ldots, c^i_{\rm EH}(K)$ are $i$ such elements, and because $c^i_{\rm EH}(K) \geq c^j_{\rm EH}(T)$, one has that $c^1_{\rm EH}(T), \ldots c^j_{\rm EH}(T) \leq c^i_{\rm EH}(K)$ and hence there are additional $j$ such elements and overall there are at least $i+j=k$ such elements.
	Now for the case $\min_{i+j=k} \max \{c^i_{\rm EH}(K),c^j_{\rm EH}(T)\} \leq M_k(K,T),$ 
	assume without loss of generality that the k-th element is $c^i_{\rm EH}(K)$ for some $i$. Obviously one has that $i\leq k$. Put $j=k-i$. It is enough to show that $c^i_{\rm EH}(K) \geq c^j_{\rm EH}(T)$. If $c^j_{\rm EH}(T) > c^i_{\rm EH}(K)$ then there are less than $k$ elements that are smaller than $c^i_{\rm EH}(K)$. Indeed, $c^{i+1}_{\rm EH}(K) > c^i_{\rm EH}(K)$, and there are at most $j-1$ elements $c^1_{\rm EH}(T),\ldots,c^{j-1}_{\rm EH}(T)$ that are less than $c^i_{\rm EH}(K)$ and hence there are at most $i+j-1<k$ such elements. Hence $c^i_{\rm EH}(K) \geq c^j_{\rm EH}(T)$ as required.

	The next case to consider is $\max_{i+j=k+1} \min \{c^i_{\rm EH}(K),c^j_{\rm EH}(T)\} \leq M_k(K,T)$.
	Let $i,j$ be the maximizers in the above inequality, and assume without loss of generality that $c^i_{\rm EH}(K) \leq c^j_{\rm EH}(T)$.
	Then there are at most $k-1$ elements that are strictly smaller than $c^i_{\rm EH}(K)$. Indeed, $c^1_{\rm EH}(K),\ldots,c^{i-1}_{\rm EH}(K)$ are $i-1$ such elements, and $c^1_{\rm EH}(T),\ldots, c^{j-1}_{\rm EH}(T)$ are $j-1$ such elements, so overall there are at most $i-1+j-1=k-1$ such elements and $c^i_{\rm EH}(K) \leq M_k(K,T)$.
	Finally let us deal with the case $\max_{i+j=k+1} \min \{c^i_{\rm EH}(K),c^j_{\rm EH}(T)\} \geq M_k(K,T).$
	Assume without loss of generality that the $k$-th element is $M_k(K,T)=c^i_{\rm EH}(K)$ for some $i$.
	Put $j=k+1-i$. Then $c^j_{\rm EH}(T) \geq c^i_{\rm EH}(K)$, otherwise there are $j$ elements $c^1_{\rm EH}(T),\ldots,c^j_{\rm EH}(T)$ that are strictly smaller than $c^i_{\rm EH}(K)$ plus $i-1$ elements $c^1_{\rm EH}(K),\ldots, c^{i-1}_{\rm EH}(K)$ that are strictly smaller than $c^i_{\rm EH}(K)$, and hence there are $i-1+j=k$ elements that are strictly smaller than $c^i_{\rm EH}(K)$, which means that $c^i_{\rm EH}(K) > M_k(K,T)$. Therefore,  $$\max_{i+j=k+1} \min \{c^i_{\rm EH}(K),c^j_{\rm EH}(T)\} \geq \min \{c^i_{\rm EH}(K),c^j_{\rm EH}(T)\} = c^i_{\rm EH}(K) = M_k(K,T),$$
	and the proof of the lemma is thus complete. 
\end{proof}

\noindent Pazit Haim-Kislev \\
\noindent School of Mathematical Sciences, Tel Aviv University, Israel \\
\noindent e-mail: pazithaim@mail.tau.ac.il
\vskip 10pt
\noindent Yaron Ostrover \\
\noindent School of Mathematical Sciences, Tel Aviv University, Israel \\
\noindent e-mail: ostrover@tauex.tau.ac.il

\end{document}